\documentclass[a4paper,12pt,centertags]{amsart}
\usepackage[utf8]{inputenc}
\usepackage{mathrsfs}
\usepackage{bbm}
\usepackage[a4paper,centering]{geometry}
\usepackage[pdfdisplaydoctitle,colorlinks,urlcolor=blue,linkcolor=blue,citecolor=blue]{hyperref}
\usepackage{mathscinet}
\newtheorem{theorem}{Theorem}[section]
\newtheorem{lemma}[theorem]{Lemma}
\newtheorem{corollary}[theorem]{Corollary}
\theoremstyle{definition}
\newtheorem{assumption}[theorem]{Assumption}
\theoremstyle{remark}
\newtheorem{remark}[theorem]{Remark}
\newtheorem{example}[theorem]{Example}
\numberwithin{equation}{section}
\DeclareMathOperator{\Div}{div}
\DeclareMathOperator{\Tr}{Tr}
\DeclareMathOperator{\Span}{span}
\newcommand{\e}{\operatorname{e}}	
\newcommand{\N}{\mathbf{N}}		
\newcommand{\R}{\mathbf{R}}		
\newcommand{\E}{\mathbb{E}}		
\newcommand{\Prob}{\mathbb{P}}		
\newcommand{\scal}[1]{\langle #1 \rangle}
\newcommand{\uno}{\mathbbm{1}}		
\newcommand{\malder}{\mathcal{D}}	
\newcommand{\leb}{\mathscr{L}}
\newcommand{\cov}{\mathcal{C}}
\usepackage{xspace}
\newcommand{\memo}[1]{
  \ensuremath{\framebox{\tiny\textbf{\kern-2pt\textsf{#1}}\kern-2pt}}\xspace
}
\hypersetup{%
  pdftitle={Existence of densities for 3D Navier--Stokes with noise},
  pdfauthor={A. Debussche, M. Romito},
  pdfcreator={M. Romito}
}
\begin{document}
  \title[Existence of densities for 3D Navier--Stokes with noise]%
        {Existence of densities for the 3D Navier--Stokes equations driven by Gaussian noise}
  \author[A. Debussche]{Arnaud Debussche}
    \address{\'Ecole Normale sup\'erieure de Cachan\\ Antenne de Bretagne\\ Campus de Ker Lann\\ Avenue Robert Schuman\\ F-35170 Bruz, France}
    \email{\href{mailto:arnaud.debussche@bretagne.ens-cachan.fr}{arnaud.debussche@bretagne.ens-cachan.fr}}
    \urladdr{\url{http://www.bretagne.ens-cachan.fr/math/people/arnaud.debussche}}
  \author[M. Romito]{Marco Romito}
    \address{Dipartimento di Matematica, Universit\`a di Pisa, Largo Bruno Pontecorvo 5, I--56127 Pisa, Italia}
    \email{\href{mailto:romito@dm.unipi.it}{romito@dm.unipi.it}}
    \urladdr{\url{http://www.dm.unipi.it/pages/romito}}
  \subjclass[2010]{Primary 76M35; Secondary 60H15, 60G30, 35Q30}
  \keywords{Density of laws, Navier-Stokes equations, stochastic partial differential
  equations, Besov spaces.}
  \date{March 1, 2012}
  \begin{abstract}
   We prove three results on the existence of densities for the laws of finite dimensional 
   functionals of the solutions of the stochastic Navier-Stokes equations in dimension $3$.  
   In particular, under very mild assumptions on the noise, we prove that finite dimensional 
   projections of the solutions have densities with respect to the Lebesgue measure which have
   some smoothness when measured in a Besov space. This is proved thanks to a new argument
   inspired by an idea introduced in \cite{FouPri10}
  \end{abstract}
\maketitle
\section{Introduction}

Our aim in this article is to give informations on the law of the solutions of the
stochastic Navier-Stokes equations in dimension three. The equations have the form
\begin{equation}\label{e:nse}
  \begin{cases}
    \dot u
      - \nu\Delta u
      + (u\cdot\nabla)u
      + \nabla p
      = \dot\eta,\\
    \Div u = 0,
  \end{cases}
\end{equation}
on a bounded open set $\mathcal O$. Here $u$ is the velocity, $p$ the pressure and 
$\nu$ the viscosity of an incompressible fluid in the region $\mathcal O$. The
equations are supplemented with an initial data and suitable boundary conditions.
This equation has been the subject of intense researches, a survey can be found
in the reference \cite{Fla08}.
The forcing term $\dot\eta$ in the problem above is a Gaussian noise which is
white in time and depending on space (precise definitions will be given later).
Under suitable assumptions on the noise, it is known that there exist weak
solutions, both in the probabilistic and PDE sense. Their uniqueness is a completely
open problem. Also, there exists a unique strong solution on small time intervals.
The situation is therefore similar to the theory of the deterministic Navier-Stokes
equations. 

However, some more informations can be obtained if the noise is sufficiently non
degenerate. It has been shown in \cite{DapDeb03,DebOda06,FlaRom06,FlaRom07,FlaRom08}
that it is possible to construct Markov solutions which depends continuously on
the initial data. This indicates that the noise can be helpful to obtain more results
in the stochastic case.

It is thus important to understand more deeply the implications of the addition of
noise. In this article, we investigate the existence of densities of the law of the
solutions. Existence of densities can be considered as a sort of smoothness of the
solution, albeit a purely probabilistic one.

A first classical difficulty is that the solutions live in an infinite dimensional
space and that no standard reference measure, such as the Lebesgue measure in finite
dimension, exists. It is tempting to use other measures as reference measure and,
in \cite{BogDapRoc96, DapDeb04, MatSui05} for instance, it is proved that for some
equations the solutions have densities with respect to a Gaussian measure.
Unfortunately, these results do not even cover the stochastic Navier-Stokes equations
in dimension two. Another possibility is to try to prove existence of densities for
finite dimensional functionals of the solutions. The problem of existence of densities
of the solutions evaluated at a fixed spatial point has been already studied
by several authors and we point to \cite{Nua06} for references. In the case
of the two dimensional Navier-Stokes equation, finite dimensional projections
of the solutions are studied in \cite{MatPar06}, where using Malliavin 
calculus, it is proved that there exist smooth densities.

Unfortunately, it seems hopeless to use Malliavin calculus for the three dimensional
Navier-Stokes equations. Indeed, it is not even possible to prove that the solutions
are Malliavin differentiable. The reason for this is immediately apparent once one
notices that the equation satisfied by the Malliavin derivative is essentially
the linearisation of Navier--Stokes, and any piece of information on that
equation could be used with much more proficiency for proving well--posedness.

In this article, we propose three different approaches to the problem which give different
results. First, in Section~\ref{s:markov} we prove existence of densities
under strong regularity and non--degeneracy assumptions on the noise
(see Assumption~\ref{a:markovnoise}). Due to the stronger assumptions,
we are able to prove existence of densities for any smooth enough map
of the solution with values in a finite dimensional space. In
Section~\ref{s:girsanov} we prove, by means of Girsanov's theorem,
existence of densities for projections onto sub--spaces spanned by a finite
number of Fourier modes, under the sole assumption that the covariance
is injective (hence without regularity assumptions). A by--product of
this technique is that the same statement holds also for the projection
(onto the same sub--spaces) of the joint law of the solution evaluated at
a finite number of time instants (see Remark~\ref{r:piutempi}).
Finally, in Section~\ref{s:besov}, under the same assumptions on the
covariance of the Girsanov case, we prove again existence of densities
with a completely different method, which extends an idea of \cite{FouPri10}.
This allows to show regularity of densities in the class of Besov spaces
and to prove that they are in $L^p$ spaces, for some $p>1$, whereas the first
two methods only provide existence of a density in $L^1$. The regularity can
be even slightly improved for stationary solutions.

We believe that these result will helpful for future research on the
Navier--Stokes equations. Recall for instance that the results of
\cite{MatPar06} have been a crucial step towards the fundamental result
of \cite{HaiMat06}. Moreover, it seems that some ideas introduced in this
article are new and can be used to prove existence of densities in other
situations, we refer to \cite{DebFou12,Fou12} for two applications of the
method used in Section \ref{s:besov}.
\subsection*{Acknowledgements}

The second author wishes to thank the kind hospitality of the {ENS}
de Cachan  Bretagne, where part of this work was realised.
\section{Preliminaries}\label{s:notation}

We consider problem \eqref{e:nse} with either periodic boundary conditions
on the three--dimensional torus $\mathcal O=[0,2\pi]^3$ or Dirichlet
boundary conditions on a smooth domain $\mathcal O\subset \R^3$. 
Let $H$ be the closure in $L^2=L^2(\mathcal O;\R^3)$ of the space
of smooth vector fields with divergence zero and satisfying the boundary
conditions (either periodic or Dirichlet). The inner product in $H$ is
denoted by $\scal{\cdot,\cdot}$ and its norm by $\|\cdot\|_H$. The space
$V$ is the closure of the same space with respect to the $H^1=H^1(\mathcal O;\R^3)$
norm. Denote by $\Pi$ the Leray projector, namely the orthogonal projector
of $L^2$ onto $H$.

Let $A=-\Pi\Delta$, with domain $D(A)=V\cap H^2(\mathcal O,\R^3)$, the
\emph{Stokes} operator and let $(\lambda_k)_{k\geq1}$ and $(e_k)_{k\geq1}$
be the eigenvalues and the corresponding orthonormal basis of eigenvectors of $A$.

The bi--linear operator $B:V\times V\to V'$ is the (Leray) projection
of the non--linearity $(u\cdot\nabla)u$ onto divergence--free vector fields:
\[
  B(u,v)
    = \Pi\left( u\cdot\nabla v\right),
      \qquad u,v\in V,
\]
and $B(u)=B(u,u)$. The operator $B$ can be easily extended to more general $u,v$.
We recall that the following properties hold,
\begin{equation}\label{e:B}
  \scal{u_1, B(u_2,u_3)} = - \scal{u_3, B(u_2,u_1)}
    \qquad\text{and}\qquad
  \scal{u_!, B(u_2,u_1)} = 0,
\end{equation}
for all $u_1,u_2,u_3$ such that the above expressions make sense. Moreover,
there is a constant $c>0$ such that
\begin{equation}\label{e:basic}
  \|A^{\frac12}B(u_1,u_2)\|_H
    \leq c\|Au_1\|_H\|Au_2\|_H,
      \qquad u_1,\, u_2\in D(A).
\end{equation}
(see for instance \cite{ConFoi88}).
We refer to Temam \cite{Tem95} for a detailed account of all the above definitions.

We assume that the noise $\dot\eta$ in \eqref{e:nse} is of white noise type and can 
be described as follows.  Consider a filtered probability space
$(\Omega,\mathcal F,\mathbb P,\{\mathcal F_t\}_{t\geq 0})$ and a cylindrical
Wiener process $W=\sum_{i\in \N} \beta_i q_i$, where $(\beta_i)_{i\in \N}$
is a family of independent Wiener processes adapted to $(\mathcal F_t)_{t\ge 0}$
and $(q_i)_{i\in\N}$ is an orthonormal basis of $H$ (see \cite{DapZab92}).
The noise $\dot\eta$ is coloured in space with a covariance operator
$\cov\in\mathcal{L}(H)$ which is positive and symmetric. It is thus of the form 
$ \dot\eta=\cov^{\frac12}\frac{dW}{dt}$.
We assume that $\cov$ is trace--class and we denote by $\sigma^2 = \Tr(\cov)$
its trace. Finally, consider the sequence $(\sigma_k^2)_{k\in\N}$ of eigenvalues
of $\cov$. It is no loss of generality to assume that  $(q_k)_{k\in\N}$ is the
orthonormal basis in $H$ of eigenvectors of $\cov$: $\cov q_k = \sigma_k^2 q_k$.
Further assumptions on $\cov$ will be considered in the following sections.

With the above notations, we can recast problem~\eqref{e:nse} as an
abstract stochastic equation
\begin{equation}\label{e:nseabs}
  du + \nu Au + B(u) = \cov^\frac12\,dW,
\end{equation}
supplemented with an initial condition $u(0)=x\in H$. It is classical that for
any $x\in H$ there exist a martingale solution of this equation. More precisely,
there exists a filtered probability space $(\widetilde\Omega,
\widetilde{\mathcal F},\widetilde{\mathbb P},
\{\widetilde{\mathcal F}_t\}_{t\ge 0})$, a cylindrical Wiener process
$\widetilde W$ and a process $u$ with trajectories in 
$C(\R^+;D(A^{-1}))\cap L^\infty_{loc}(\R^+,H)\cap L^2_{loc}(\R^+;V)$
adapted to $(\widetilde{\mathcal F}_t)_{t\ge 0}$ such that the above
equation is satisfied with $\widetilde W$ replacing $W$. See for instance
\cite{Fla08} and the references therein for further details.

The existence of martingale solutions is equivalent to the existence
of a solution of the following martingale problem. We say that a probability
measure $\Prob_x$ on  $C(\R^+;D(A^{-1}))$ is a solution of the martingale
problem associated to equation \eqref{e:nseabs}  with initial condition
$x\in H$ if
\begin{itemize}
  \item $\Prob_x[L^\infty_{loc}(\R^+,H)\cap L^2_{loc}(\R^+;V)] = 1$,
  \item for each $\phi\in D(A))$, the process 
    \[
     \scal{\xi_t -\xi_0,\phi}
       + \int_0^t \scal{\xi_s, A\phi}
       - \scal{B(\xi_s,\phi),\xi_s}\,ds
    \]
    is a continuous square integrable martingale with quadratic variation
    equal to $t\|\cov^{1/2}\phi\|_H^2$,
  \item the marginal of $\Prob_x$ at time $0$ is the Dirac mass $\delta_x$ at $x$,
\end{itemize}
where in the formulae above $(\xi_t)_{t\geq0}$ is the canonical process
on the path space $C(\R^+;D(A^{-1}))$.

The law of a martingale solutions is a solution of the martingale problem.
Conversely, given a solution of the martingale problem, it is not difficult
to prove that the canonical process provides a martingale solution (see
\cite{Fla08} for details).
 
If $K$ is an Hilbert space, we denote by $\mathcal{L}(K)$ the space
of linear bounded operators from $K$ into itself, by $\pi_F:K\to K$
the orthogonal projection of $K$ onto a subspace $F\subset K$, and by
$\Span[x_1,\dots,x_n]$ the subspace of $K$ generated by
its elements $x_1,\dots,x_n$. Also $\mathcal{B}(K)$ is the set of
Borel subsets of $K$.

We shall use the symbol $\leb_d$ to denote the Lebesgue measure
on $\R^d$ and the symbol $\leb_F$ to denote the Lebesgue measure
on a finite dimensional space $F$ induced by the representation by
a basis. Finally, given a measure $\mu$ and a measurable map $f$,
we denote by $f_\#\mu$ the image measure of $\mu$ through $f$,
namely $(f_\#\mu)(E) = \mu(f^{-1}(E))$.
\section{Existence of densities with non--degenerate noise: the Markovian case}
  \label{s:markov}

In this section we shall consider the following assumptions on the covariance.
\begin{assumption}\label{a:markovnoise}
  There are $\epsilon>0$ and $\delta\in(1,\tfrac32]$ such that
  \begin{itemize}
    \item $\Tr(A^{1+\epsilon}\cov)<\infty$,
    \item $\cov^{-\frac12}A^{-\delta}\in\mathcal{L}(H)$.  
  \end{itemize}
  
  For example, $\cov = A^{-\alpha}$ with $\alpha\in(\tfrac52,3]$
  satisfies the above assumptions.
\end{assumption}
Under the above assumptions, it has been proved in \cite{DapDeb03,DebOda06}
and in \cite{FlaRom06,FlaRom07,FlaRom08} using a different method that
there exists a family, indexed by the initial condition, of Markov solutions.

We say that $P(\cdot,\cdot,\cdot):[0,\infty)\times D(A)\times\mathcal{B}(D(A))
\to [0,1]$ is a Markov kernel in $D(A)$ of transition probabilities associated
to equation \eqref{e:nseabs} if $P(\cdot,\cdot,\Gamma)$ is Borel measurable
for every $\Gamma\in\mathcal{B}(D(A))$, $P(t,x,\cdot)$ is a probability
measure on $\mathcal{B}(D(A))$ for every $(t,x)\in [0,\infty)\times D(A)$,
the Chapman--Kolmogorov equation
\[
  P(t +s,x,\Gamma)
    = \int_{D(A)} P(t,x,dy)P(s,y,\Gamma)
\]
holds for every $t,s\geq0$, $x\in D(A)$, $\Gamma\in\mathcal{B}(D(A))$,
and for every $x\in D(A)$ there is a solution $\Prob_x$ of the martingale
problem associated to equation \eqref{e:nseabs} with initial condition $x$
such that $P(t,x,\Gamma) = \Prob_x(\xi_t\in\Gamma)$ for all $t\ge 0$.

Moreover, $P(\cdot,x,\cdot)$ and $\Prob_x$, solution of the martingale problem,
can be defined for all $x\in H$ and the Chapman--Kolmogorov equation holds almost
everywhere in $s$. More precisely, for every $x\in H$ and every $t\ge 0$, there
is a set $I\subset \R^+$ such that the Chapman--Kolmogorov equation 
holds for all $\Gamma\in \mathcal{B}(H)$. Also  
\[
  \mathcal{P}_t\varphi (x)
    = \E^{\Prob_x}[\varphi(\xi_t)],
      \qquad t\geq 0, x\in H,
\]
defines a transition semigroup $(\mathcal{P}_t)_{t\geq0}$. It turns
out that this transition semigroup has the strong Feller property, that
is $\mathcal{P}_t\phi$ is continuous on $D(A)$ if $\phi$ is merely
bounded measurable. Several other results can be found in the above references.
In the following, by a Markov solution $(\Prob_x)_{x\in H}$ we mean
a family of probability measures on $C(\R^+;D(A))$ associated to
transition probabilities satisfying the above properties.

Let $f:D(A)\to\R^d$ be $C^1$ and define, for our purposes,
a \emph{singular point} $x$ of $f$ as a point where the range of $Df(x)$ is
a proper subspace of $\R^d$. The following result is proved in Sections
\ref{s2.1}, \ref{s2.2} and \ref{s2.3} below.
\begin{theorem}\label{t:markov}
  Let Assumption \ref{a:markovnoise} hold and let $f:D(A)\to\R^d$
  be a map such that the set of singular points (as defined above)
  is not dense. Given an arbitrary Markov solution $(\Prob_x)_{x\in H}$,
  let $u(\cdot,x)$ be a random field with distribution $\Prob_x$. Then
  for every $t>0$ and every $x\in H$, the law of the random variable
  $f(u(t;x))$ has a density with respect to the Lebesgue measure
  $\leb_d$ on $\R^d$, which is almost everywhere positive.
\end{theorem}
\begin{example}
  We give a few significant examples for the previous theorem.
  \begin{itemize}
    \item Functions such as $f(u(t)) = \|u(t)\|_H^2$, as well as any
      other norm which is well defined in $D(A)$ admit a density with
      respect to the Lebesgue measure.
    \item In view of the results of the following sections, consider
      the case where $f\approx \pi_F$, where $F$ is a finite
      dimensional subspace of $D(A)$, $\pi_F$ is the projection onto
      $F$ and $f$ is given as $f(x)=(\scal{\cdot,f_1},\dots,\scal{\cdot,f_d})$,
      where $f_1,\dots,f_d$ is a basis of $F$. Then the image measure
      $(\pi_F)_\#\Prob_x$ is absolutely continuous with respect to the
      Lebesgue measure of $F$.
    \item Given points $y_1,\dots,y_d\in\R^3$ (or in the corresponding
      bounded domain in the Dirichlet boundary condition case), the map 
      $x\mapsto(x(y_1),\dots,x(y_d))$, defined on $D(A)$, clearly
      meets the assumptions of the previous theorem, and hence has a density
      on $\R^d$, since the elements of $D(A)$ are continuous functions
      by Sobolev's embeddings.
  \end{itemize}
\end{example}
\begin{remark}
  In view of \cite{RomXu09,AlbDebXu10}, the same result holds true under
  a slightly weaker assumption of non--degeneracy on the covariance of
  the driving noise. In few words, a finite number of components of the
  noise can be zero.
\end{remark}
By the results of \cite{DapDeb03,DebOda06,Oda07} or \cite{Rom08}, each
Markov solution converges to its unique invariant measure.
The following result is a straightforward consequence of the theorem above.
\begin{corollary}
  Under the same assumptions of Theorem~\ref{t:markov}, given a Markov
  solution $(\Prob_x)_{x\in H}$, denote by $\mu_\star$ its
  invariant measure. Then the image measure $f_\#\mu_\star$ has
  a density with respect to the Lebesgue measure on $\R^d$.
\end{corollary}
\begin{proof}
  If $(P(t,x,\cdot))_{t\geq0,x\in H}$ is the corresponding Markov
  transition kernel and $E\subset\R^d$ has Lebesgue measure
  $\leb_d(E)=0$, then by Theorem \ref{t:markov}
  $P(t,x, f^{-1}E)=f_\#P(t,x, E)=0$ for each $x\in H$ and $t>0$.
  Then, by Chapman--Kolmogorov, 
  \[
    f_\#\mu_\star(E)
      = \int_{D(A)} P(t,x, f^{-1}E)\,\mu_\star(dx)
      = 0,
  \]
  since $\mu_\star(D(A))=1$.
\end{proof}
\subsection{Reduction to the local smooth solution}
\label{s2.1}

Let $\chi\in C^\infty(\R)$ be a function such that $0\leq \chi\leq 1$,
$\chi(s)=1$ for $s\leq 1$ and $\chi(s)=0$ for $s\geq 2$, and set
for every $R>0$, $\chi_R(s)=\chi(\tfrac{s}{R})$. Set
\begin{equation}\label{e:BR}
  B_R(v)
    = \chi_R(\|Av\|_H^2)B(v)
\end{equation}
and denote by $u_R(\cdot;x)$ the solution of
\begin{equation}\label{e:nsesmooth}
  du_R + \bigl(\nu A u_R + B_R(u_R)\bigr)\,dt = \cov^{\frac12}dW.
\end{equation}
with initial condition $x\in D(A)$.
Existence, uniqueness as well as several regularity properties
are proved in \cite[Theorem 5.12]{FlaRom08}. We denote by $P_R(\cdot,\cdot,\cdot)$ 
and $\Prob^R_x$ the associated transition probabilities and laws of the solutions. 

Define
\begin{equation}\label{e:blowup}
  \tau_R
    = \inf\{t\geq0: \|Au_R(t)\|_H^2\geq R\},
\end{equation}
then again by \cite[Theorem 5.12]{FlaRom08} it follows that $\tau_R>0$
with probability one if $\|Ax\|_H^2<R$ and that \emph{weak--strong
uniqueness} holds: every martingale solution of \eqref{e:nseabs} starting
at the same initial condition $x$ coincides with $u_R(\cdot;x)$ up to
time $t$ on the event $\{\tau_R>t\}$, for every $t>0$.
\begin{lemma}\label{l:reduction}
  Let Assumption~\ref{a:markovnoise} be true and let $f:D(A)\to\R^d$ be
  a measurable function. Assume that for every $x\in D(A)$, $t>0$ and
  $R\geq1$ the image measure $f_\#P_R(t,x,\cdot)$ of
  the transition density $P_R(t,x,\cdot)$ corresponding to problem
  \eqref{e:nsesmooth} is absolutely continuous with respect to the Lebesgue
  measure $\leb_d$ on $\R^d$. Then the probability measure $f_\#P(t,x,\cdot)$
  is absolutely continuous with respect to $\leb_d$ for every $x\in H$
  every $t>0$ and every Markov solution $(\Prob_x)_{x\in H}$.
\end{lemma}
\begin{proof}
  Fix a Markov solution $(\Prob_x)_{x\in H}$ and denote by $P(t,x,\cdot)$
  the associated transition kernel.

  \emph{Step 1}. We prove that each solution is concentrated on $D(A)$
  at every time $t>0$, for every initial condition $x$ in $H$.
  By \cite[Lemma 3.7]{Rom08}
  \[
    \E^{\Prob_x}\Bigl[\int_0^t \|A\xi_s\|^\delta\,ds\Bigr]
      <\infty,
  \]
  for some $\delta >0$. Thus $P(s,x,D(A))=1$  for almost every
  $s\in [0,t]$. Recall that, for $z\in D(A)$ and $r\ge 0$,
  $P(r,z,D(A))=1$. We deduce that $P(t-s,y,D(A))=1$,
  $P(s,x,\cdot)$--{a.~s.} for almost every $s\in [0,t]$.
  Since the Chapman-Kolmogorov equation holds for almost every
  $s$, we have:
  \[
    P(t,x,D(A))
      =\frac1t \int_0^t \Bigl(\int P(t-s,y,D(A))\,P(s,x,dy)\Bigr)\,ds
      = 1.
  \]

  \emph{Step 2}. Given $x\in D(A)$, $s>0$ and $B\subset D(A)$ measurable,
  we prove the following formula, 
  \[
    \bigl|P(s,x,B) - P_R(s,x,B)\bigr|
      \leq 2\Prob_x[\tau_R\leq s].
  \]
  Indeed, by weak--strong uniqueness,
  \[
    \begin{aligned}
      P(s,x,B)
        & =  \E^{\Prob_x}[\xi_s\in B, \tau_R>s]
           + \E^{\Prob_x}[\xi_s\in B, \tau_R\leq s]\\
        & =  P_R(s,x,B)
           + \E^{\Prob_x}[\xi_s\in B, \tau_R\leq s]
           - \E^{\Prob_x^R}[\xi_s\in B, \tau_R\leq s].
    \end{aligned}
  \]
  Hence the first side of the inequality holds. The other side follows
  in the same way.

  \emph{Step 3}. We prove that the lemma holds if the initial condition
  is in $D(A)$. Let $B$ be such that $\leb_d(B)=0$, hence
  $P_R(t,x,f^{-1}(B))=0$ for all $t>0$, $x\in D(A)$ and $R\geq1$,
  then
  \[
    \begin{aligned}
      P(t+\epsilon,x,f^{-1}(B))
        &= \int P(\epsilon,y,f^{-1}(B))\,P(t,x,dy)\\
        &\leq 2\int_{\{\|Ay\|_H<R\}} \Prob_x[\tau_R\leq\epsilon]\,P(t,x,dy)
              + 2 P(t,x,\{\|Ay\|_H\geq R\}),
    \end{aligned}
  \]
  Since by \cite[Proposition 11]{FlaRom07}, $\Prob_x[\tau_R\leq s]\downarrow0$
  as $s\downarrow 0$ if $\|Ax\|_H<R$, by first taking the limit as
  $\epsilon\downarrow0$ and then as $R\uparrow\infty$, we deduce, using also
  the first step, that $P(t+\epsilon,x,f^{-1}(B))\to0$ as $\epsilon\downarrow0$.
  On the other hand, by \cite[Lemma 3.1]{Rom08a},
  $P(t+\epsilon,x,f^{-1}(B))\to P(t,x,f^{-1}(B))$, hence
  $P(t,x,f^{-1}(B))=0$.
  
  \emph{Step 4}. We finally prove that the lemma holds with initial
  conditions in $H$.
  We know that $P(t,x,f^{-1}(B)) = 0$ for all $t>0$ and
  $x\in D(A)$ if $\leb_d(B)=0$. If $x\in H$ and $s>0$ is a time such
  that the {a.~s.} Markov property holds, then
  \[
    P(t,x,f^{-1}(B))
      = \int P(t-s,y,f^{-1}(B))\,P(s,x,dy)
      = 0,
  \]
  since $P(s,x,D(A))=1$ by the first step.
\end{proof}
\subsection{Absolute continuity for the truncated problem}
\label{s2.2}

We now show that the law of $f(u_R(t;x))$ has a density with respect
to the Lebesgue measure on $\R^d$ for every $R>0$, $x\in D(A)$ and
$t>0$. We use Theorem~2.1.2 of \cite{Nua06}
  
Let $x\in D(A)$, $t>0$ and $R>0$. It is standard to prove that
$u_R(t;x)$ has Malliavin derivatives and that
$\malder^s u_R(t;x)\cdot q_k=\sigma_k \eta_k(t,s;x)$,
for $s\leq t$, where $\eta_k$ is the solution of
\begin{equation}\label{e:eta}
  \begin{cases}
    d_t\eta_k
      + \nu A\eta_k
      + DB_R(u_R)\eta_k
      = 0,\\
    \eta_k(s,s;x) = q_k,
  \end{cases}
  \quad t\geq s
\end{equation}
$B_R$ is defined in \eqref{e:BR}, so that its derivative along
a direction $\theta$ is given as
\[
  DB_R(v)\theta
    =   \chi_R(\|Av\|_H^2)\bigl(B(\theta,v)+B(v,\theta)\bigr)
      + 2\chi_R'(\|Av\|_H^2)\scal{Av,A\theta}_H B(v,v),
\]
and we recall that $(q_k,\sigma_k)_{k\in\N}$ is the system
of eigenvectors and eigenvalues of the covariance $\cov$ of the
noise. Standard estimates imply that 
\[
  u_R(t;x)\in
    \mathbb{D}^{1,2}(D(A))
      = \Bigl\{u:\E\bigl[\|Au\|_H^2\bigr]
          +\sum_{k=0}^\infty\E\int_0^t\|\malder^s u\cdot q_k\|_{D(A)}^2<\infty
        \Bigr\}.
  \]
Moreover, for every $k\in\N$ and $x\in D(A)$, the function $\eta_k(t,s;x)$
is continuous in both variables $s\in [0,\infty)$ and $t\in[s,\infty)$.

By the chain rule for Malliavin derivatives, the Malliavin matrix
$\mathcal{M}^f(t)$ of $f(u_R(t;x))$ is then given by
\[
  \begin{aligned}
    \mathcal{M}_{ij}^f(t)
      & = \sum_{k=1}^\infty \int_0^t
            \bigl(Df_i(u_R(t;x))\malder^s u_R(t;x)\cdot q_k\bigr)
            \bigl(Df_j(u_R(t;x))\malder^s u_R(t;x)\cdot q_k\bigr)\,ds\\
      & = \sum_{k=1}^\infty \sigma_k^2\int_0^t
            \bigl(Df_i(u_R(t;x))\eta_k(t;s,x)\bigr)
            \bigl(Df_j(u_R(t;x))\eta_k(t;s,x)\bigr)\,ds
  \end{aligned}
\]
for $i,j=1,\dots,d$, where $f=(f_1,\dots,f_d)$. 

To show that $\mathcal{M}^f(t)$ is invertible {a.~s.}, it is sufficient
to show that if $y\in \R^d$ and
\[
  \scal{\mathcal{M}^f(t)y,y}
    = \sum_{k=1}^\infty \sigma_k^2\int_0^t\Bigl|\sum_{i=1}^dDf_i(u_R(t;x))\eta_k(t;s,x)y_i\Bigr|^2\,ds
\]
is zero, then $y=0$. This is clearly true, since if
$\scal{\mathcal{M}^f(t)y,y}=0$, then
\[
  \sum_{i=1}^d y_i Df_i(u_R(t;x))\eta_k(t;s,x)
    = 0,
      \qquad\Prob-a.s.,
\]
for all $k\in\N$ and {a.~e.} $s\leq t$. By continuity, the above equality
holds for all $s\leq t$. In particular for $s=t$ this yields
\[
  \sum_{i=1}^d y_i Df_i(u_R(t;x))q_k
    = 0,
     \quad \Prob-a.s.,
\]
for all $k\geq1$. Under our assumptions on the covariance, the support 
of the law of $u_R(t;x)$ is the full space $D(A)$ (this follows from
Lemma C.2 and Lemma C.3 of \cite{FlaRom08}). Hence, $u_R(t;x)$ belongs
to the set of non singular points of $f$ with positive probability. 
We know that $(q_k)_{k\geq1}$ is a basis of $H$, hence the family
of vectors $(Df_1(u_R(t;x))q_k,\dots,Df_d(u_R(t;x))q_k)_{k\geq1}$
spans all $\R^d$ with positive probability, and in conclusion $y = 0$.
\subsection{Proof of Theorem~\ref{t:markov}}
\label{s2.3}

Fix an initial condition $x\in H$ and a time $t>0$, and consider a
finite--dimensional map $f:D(A)\to\R^d$ satisfying the assumptions
of Theorem~\ref{t:markov}. Gathering the two previous sections, we know that
$f(u(t;x))$ has a density with respect to the Lebesgue measure on $\R^d$.

Finally, to prove that the density is almost everywhere positive, it is
sufficient to recall that under the assumptions of the theorem each
Markov solution is irreducible (see for instance \cite{Fla97} or
\cite[Proposition 6.1]{FlaRom08}).
\section{Existence of densities with non--degenerate noise: Girsanov approach}
  \label{s:girsanov}

In this section we shall consider the following assumptions for the noise.
\begin{assumption}\label{a:girsanovnoise}
  We assume that
  \begin{itemize}
    \item the covariance $\cov\in\mathcal{L}(H)$ is of trace--class,
    \item $\ker(\cov) = \{0\}$.
  \end{itemize}
\end{assumption}
We consider solutions of  \eqref{e:nse} obtained by Galerkin approximations.
Given an integer $N\geq 1$, consider the sub--space
$H_N = \Span[e_1,\dots,e_N]$ and denote by $\pi_N = \pi_{H_N}$
the projection onto $H_N$. It is standard (see for instance \cite{Fla08})
to verify that the problem
\begin{equation}\label{e:galerkin}
  du^N + \bigl(\nu Au^N + B^N(u^N))\,dt = \pi_N\cov^{\frac12}dW,
\end{equation}
where $B^N(\cdot) = \pi_N B(\pi_N\cdot)$, admits a unique strong
solution for every initial condition $x^N\in H_N$. Moreover,
\begin{equation}\label{e:Gbound}
  \E\Bigl[\sup_{[0,T]}\|u^N\|_H^p\Bigr]
    \leq c_p(1+\|x^N\|_H^p),
\end{equation}
for every $p\geq1$ and $T>0$, where $c_p$ depends
only on $p$, $T$ and the trace $\sigma^2$.

If $x\in H$, $x^N = \pi_N x$ and $\Prob^N_x$ is the distribution
of the solution of the problem above with initial condition $x^N$,
then any limit point of $(\Prob^N_x)_{N\geq1}$ is a solution
of the martingale problem associated to \eqref{e:nse} with initial
condition $x$.

We prove the following result.
\begin{theorem}\label{t:girsanov}
  Fix an initial condition $x\in H$ and let $F$ be a finite dimensional
  subspace of $D(A)$ generated by the eigenvalues of $A$, namely
  $F = \Span[e_{n_1},\dots,e_{n_F}]$ for some arbitrary indexes
  $n_1,\dots,n_F$. Under Assumption~\ref{a:girsanovnoise} for every $t>0$
  the projection $\pi_F u(t)$ has a density with respect to the Lebesgue measure
  on $F$, where $u$ is \emph{any} solution of \eqref{e:nseabs} whose law
  is a limit point of the spectral Galerkin approximations defined above.
  Moreover the density is positive almost everywhere (with respect to the
  Lebesgue measure on $F$). 
\end{theorem}
\begin{proof}
  Fix $x\in H$ and let $u$ be a weak martingale solution of \eqref{e:nseabs}
  with distribution $\Prob_x$ and assume $\Prob^{N_k}_x\rightharpoonup\Prob_x$,
  where $N_k\uparrow\infty$ is a sequence of integers and for each $k$,
  $\Prob_x^{N_k}$ is solution of \eqref{e:galerkin}. Fix a time $t>0$ and
  consider the Galerkin approximation \eqref{e:galerkin} at level $N\geq n_F$.
  
  \emph{Step 1: the Girsanov density}.
  We are going to use Girsanov's theorem in the version given in
  \cite[Theorem 7.19]{LipShi01}. Let $v^N$ be the solution of
  \[
    dv^N + \bigl(\nu A v^N + B^N(v^N) - \pi_F B^N(v^N)\bigr)\,dt
      = \pi_N \cov^{\frac12}\,dW,
  \]
  with the same initial condition as $u^N$. We notice in particular that
  the projection of $v^N$ on $F$ solves a linear equation (see \eqref{e:zetaF}
  below) which is decoupled from $v^N - \pi_F v^N$. Moreover, it is easy
  to prove, with essentially the same methods that yield \eqref{e:Gbound},
  that
  \begin{equation}\label{e:Gbound2}
   \E\Bigl[\sup_{[0,T]}\|v^N\|_H^p\Bigr]<\infty.
  \end{equation}
  Note that $\sup_{\|w\|_{W^{1,\infty}}=1}\scal{v,w}$ is a norm
  on $\pi_N H$, which is therefore equivalent to the norm of $H$ on
  $\pi_NH$. We can then write: 
  \[
    \scal{\pi_N B(v),w}
      = - \scal{B(v,\pi_N w),v}
      \leq  c \|w\|_{W^{1,\infty}} \|v\|_H^2,
  \]
  therefore
  \begin{equation}\label{e:maintruc}
    \|\pi_N B(v)\|_H
      \leq c_N \|v\|_H^2,
        \qquad v\in H.
  \end{equation}
  Since the covariance $\cov^{\frac12}$ is invertible on $\pi_NH$,
  we deduce from \eqref{e:Gbound}, \eqref{e:Gbound2} that
  \[
    \int_0^t \|\cov^{-\frac12}\pi_N B(u^N)\|_H^2\,ds<\infty,
      \qquad 
    \int_0^t \|\cov^{-\frac12}\pi_N B(v^N)\|_H^2\,ds<\infty.
      \qquad \Prob-a.s.
  \]
  By Theorem 7.19 of \cite{LipShi01} the process 
  \begin{equation}\label{e:girsanov}
    G_t^N
      = \exp\Bigl(\int_0^t \scal{\cov^{-\frac12}\pi_F B(u^N), dW_s}
                -\frac12\int_0^t\|\cov^{-\frac12}\pi_F B(u^N)\|_H^2\,ds\Bigr),
  \end{equation}
  is positive, finite $\Prob$--{a.~s.} and a martingale. Moreover, under the
  probability measure $\widetilde\Prob_N(d\omega) =G_t^N \Prob_N(d\omega)$
  the process
  \[
    \widetilde W_t = W_t - \int_0^t \cov^{-\frac12}\pi_F B(u^N)\,ds
  \]
  is a cylindrical Wiener process on $H$ and $\pi_F u^N(t)$ has the same
  distribution as the solution $z^F$ of the linear problem
  \begin{equation}\label{e:zetaF}
    dz^F + \nu A z^F = \pi_F \cov^{\frac12}\,dW,
  \end{equation}
  which is independent of $N$. In particular for every measurable
  $E\subset F$,
  \[
    \Prob_N[z^F(t)\in E]
      = \widetilde\Prob_N[\pi_F u^N(t)\in E]
      = \E^{\Prob_N}\bigl[G_t^N\uno_E(\pi_F u^N(t))\bigr]
  \]
  and if $\leb_F(E)=0$, where $\leb_F$ is the Lebesgue measure on $F$,
  then $\uno_E(z^F(t))=0$. Since $G_t^N>0$, $\Prob_N$--{a.~s.},
  we have that $\uno_E\bigl(\pi_F u^N(t)\bigr)=0$, $\Prob_N$--{a.~s.},
  that is $\Prob_N[\pi_F u^N(t)\in E]=0$. In conclusion $\pi_F u^N(t)$
  has a density with respect to the Lebesgue measure on $F$.

  \emph{Step 2: passage to the limit.}
  Now consider the weak martingale solution $u$ of the infinite dimensional
  problem~\eqref{e:nseabs}. We show that $G_t^{N_k}$ is convergent. By possibly
  changing the underlying probability space and the driving Wiener process
  via the Skorokhod theorem, we can assume that there is a sequence of processes
  $(u^{N_k})_{k\geq1}$ such that $u^{N_k}$ has law $\Prob_{N_k}$ and
  $u^{N_k}\to u$ {a.~s.} in $C([0,T];H_w)$ --- where $H_w$ is the space $H$
  with the weak topology --- and in $L^2(0,T;H)$ for every $T>0$.
  In particular the sequence $(u^{N_k})_{k\geq1}$ is {a.~s.} bounded in
  $L^\infty(0,T;H)$  and thus {a.~s.} strongly convergent in $L^p(0,T;H)$
  for every $T>0$ and every $p<\infty$. This ensures that $G_t^{N_k}\to G_t$,
  {a.~s.}, for every $t$, where $G_t$ is the same as in~\eqref{e:girsanov}
  for $u$.
  
  Moreover, we notice that $G_t>0$ and finite {a.~s.}, since by
  \eqref{e:maintruc} and \eqref{e:Gbound},
  \[
    \frac12\int_0^t\|\cov^{-\frac12}\pi_F B(u)\|_H^2\,ds<\infty,
      \qquad\text{a.~s.}
  \]
  on the limit solution.

  \emph{Step 3: conclusion.}
  We show that $\pi_F u(t)$ has a density with respect to the
  Lebesgue measure on $F$. Let $E\subset F$ with $\leb_F(E)=0$, then
  for every open set $J$ such that $E\subset J$ we have by Fatou's lemma
  (notice that $\uno_J$ is lower semi--continuous with respect to the weak
  convergence in $H$ since $J$ is finite dimensional),
  \begin{multline*}
    \E[G_t\uno_E(\pi_F u(t))]
      \leq \E[G_t\uno_J(\pi_F u(t))] \leq\\
      \leq \liminf_N \E[G_t^N\uno_J(\pi_F u^N(t))]
      = \Prob[z^F(t)\in J],
  \end{multline*}
  hence $\E[G_t\uno_E(\pi_F u(t))]=0$ since $J$ can have arbitrarily small
  measure and $z^F(t)$ has Gaussian density. Again we deduce that
  $\Prob[\pi_F u(t)\in E] = 0$ from the fact that $G_t>0$.
  Finally, the fact that the density of $\pi_F u(t)$ is positive
  follows from the results of \cite{Shi07}. 
\end{proof}
\begin{remark}
  The bounds on the sequence $(G_t^N)_{N\geq1}$ are not strong enough to
  deduce a stronger convergence to $G_t$ and hence to deduce the representation
  \begin{equation}\label{e:Grep1}
    \E[\phi(z^F(t))]
      = \E\bigl[G_t\phi(\pi_F u(t))\bigr]
  \end{equation}
  in the limit, for smooth function $\phi:F\to\R$. Although this formula
  would provide a representation for the (unknown) density of $\pi_F u(t)$
  in terms of the (known) density of $z^F$, solution of \eqref{e:zetaF}, this
  would not characterise the law of $\pi_F u(t)$ by any means, since the
  factor $G_t$ which appears in the formula depends on the sub--sequence
  $(N_k)_{k\in\N}$ which ensures that $\Prob_{N_k}\rightharpoonup\Prob$.

  Vice versa, one could use the inverse density
  \[
    \widetilde G_t^N
      = \exp\Bigl(- \int_0^t \scal{\cov^{-\frac12}\pi_F B(v^N), dW_s}
           -\frac12\int_0^t\|\cov^{-\frac12}\pi_F B(v^N)\|_H^2\,ds\Bigr),
  \]
  which is also a martingale by \cite[Theorem 7.19]{LipShi01}, to get
  in the limit
  \begin{equation}\label{e:Grep2}
    \E[\phi(\pi_F u(t))]
      = \E\bigl[\widetilde G_t\phi(z^F(t))\bigr]
  \end{equation}
  but the bound~\eqref{e:Gbound2} for $v^N$ is not uniform in $N$.
\end{remark}
\begin{remark}\label{r:piutempi}
  Since our proof is based on Girsanov formula, it gives more information.
  In fact, it easily extends to show that for every $t_1,\dots,t_m$, the
  law of $(\pi_F u(t_1),\dots,\pi_F u(t_m))$ has a density with respect
  to the Lebesgue measure on $F\times\dots\times F$.
  \end{remark}
\begin{remark}\label{r:relaxed}
  In the proof of the above theorem we have actually used that the
  covariance is invertible \emph{only on} $F$. Indeed Assumption
  \eqref{a:girsanovnoise} is too restrictive and the second property
  can be replaced by the following weaker property,
  \begin{itemize}
    \item $F\cap\ker(\cov) = \{0\}$,
  \end{itemize}
  which allows a degenerate covariance.
\end{remark}
\section{Existence of densities with non--degenerate noise: bounds in Besov spaces}
  \label{s:besov}

We now show that the density found in the previous theorem has a little bit
more regularity than the one provided by the Radon--Nykodym theorem. At the
same time we provide an alternative proof of existence of the density, which
is based on an idea of~\cite{FouPri10}. We work again under the
Assumption~\ref{a:girsanovnoise}, although again Remark~\ref{r:relaxed} may
apply.

We prove in fact that the density belongs to a suitable Besov space. 
A general definition of Besov spaces $B_{p,q}^s(\R^d)$ is given by means of
Littlewood--Paley's decomposition. Here we use the equivalent definition
given in \cite[Theorem 2.5.12]{Tri83} or \cite[Theorem 2.6.1]{Tri92}
in terms of differences. Define
\[
  \begin{gathered}
    (\Delta_h^1f)(x)
      = f(x+h)-f(x),\\
    (\Delta_h^nf)(x)
      = \Delta_h^1(\Delta_h^{n-1}f)(x)
      = \sum_{j=0}^n (-1)^{n-j}\binom{n}{j} f(x+jh)
  \end{gathered}
\]
then the following norms, for $s>0$, $1\leq p\leq\infty$, $1\leq q<\infty$,
\[
  \|f\|_{B_{p,q}^s}
    = \|f\|_{L^p}
      + \Bigl(\int_{\{|h|\leq 1\}}\frac{\|\Delta_h^n f\|_{L^p}^q}{|h|^{sq}}
          \frac{dh}{|h|^d}\Bigr)^{\frac1q}
\]
and for $q=\infty$,
\[
  \|f\|_{B_{p,\infty}^s}
    = \|f\|_{L^p}
      + \sup_{|h|\leq 1}\frac{\|\Delta_h^n f\|_{L^p}}{|h|^\alpha},
\]
where $n$ is any integer such that $s<n$, are equivalent norms of $B_{p,q}^s(\R^d)$
for the given range of parameters. Note that if $1\leq p<\infty$ and $s>0$
is not an integer, then $B^s_{p,p}(\R^d) =  W^{s,p}(\R^d)$
(this is formula 2.2.2/(18) of \cite{Tri83}). We refer to \cite{Tri83,Tri92}
for a general introduction on these spaces, for their properties
and for further details on the topic.
\subsection{Besov regularity of the densities}

We prove the following result. 
\begin{theorem}\label{t:besov1}
  Fix an initial condition $x\in H$ and let $F$ be a finite dimensional subspace
  of $D(A)$ generated by the eigenvectors of $A$, namely
  $F = \Span[e_{n_1}, \dots,e_{n_F}]$ for some arbitrary indices
  $n_1,\dots,n_F$. Under Assumption \ref{a:girsanovnoise}, for every $t>0$ the
  projection $\pi_F u(t)$ has an almost everywhere positive density $f_{F,t}$ with respect
  to the Lebesgue measure on $F$, where $u$ is any solution of \eqref{e:nseabs}
  which is limit point of the spectral Galerkin approximations \eqref{e:galerkin}.

  Moreover $f_{F,t}\in B_{1,\infty}^s(\R^d)$, hence
  $f_{F,t}\in W^{s,1}(\R^d)$, for every $s\in(0,1)$, and
  $f_{F,t}\in L^p(\R^d)$ for any $p\in [1,\tfrac{d}{d-1})$,
  where $d = \dim F$.
\end{theorem}
\begin{proof}
  Let $u$ be a weak martingale solution of~\eqref{e:nseabs} with initial
  condition $x$ and distribution $\Prob_x$,
  and assume $\Prob_x^{N_k}\rightharpoonup\Prob_x$, where $N_k\uparrow\infty$
  is a sequence of integers and for each $k$ the probability measure $\Prob_x^{ N_k}$
  is a weak martingale solution of \eqref{e:galerkin} with initial condition
  $\pi_{N_k}x$. Given a finite dimensional
  space $F = \Span[e_{n_1}, \dots,e_{n_F}]$ and a time $t$, which without loss
  of generality is taken equal to $t=1$, we wish to show that the random
  variable $\pi_F u(1)$ has a density with respect to the Lebesgue measure
  on $F$ (which we identify with $\R^d$, $d = \dim F$). We first notice
  that, again by the results of \cite{Shi07}, the density will be positive
  almost everywhere.
  
  For $N\geq n_F$, let $f_N$ be the density of the random variable
  $\pi_F u^N(1)$, where $u^N$ is the solution of~\eqref{e:galerkin}.
  The existence of $f_N$ is easy to prove under Assumption
  \ref{a:girsanovnoise}. For every $\epsilon<1$, denote
  by $\eta_\epsilon = \uno_{[0,1-\epsilon]}$ the indicator function of the
  interval $[0,1-\epsilon]$. Denote by $u^{N,\epsilon}$ the solution of
  \[
    d u^{N,\epsilon}
     + \bigl(\nu A u^{N,\epsilon}
     + B(u^{N,\epsilon})
     - (1-\eta_\epsilon) \pi_F B(u^{N,\epsilon})\bigr)\,dt
     = \pi_N \cov^{\frac12}\,dW,
  \]
  where $\pi_N$ is the projection onto $\Span[e_1,\dots,e_N]$, and notice that
  $u^{N,\epsilon}(s) = u^N(s)$ for $s\leq t-\epsilon$. Moreover for
  $t\in [1-\epsilon,1]$, $v=\pi_F u^{N,\epsilon}$ satisfies
  \begin{equation}\label{e:justbefore}
    \begin{cases}
      dv + \nu \pi_F A v\,dt = \pi_F \cov^{\frac12}\,dW,\\
      v(1-\epsilon)=\pi_F u^{N,\epsilon}(1-\epsilon).
    \end{cases}
  \end{equation}
  Therefore, conditioned to $\mathcal{F}_{1-\epsilon}$,
  $\pi_F u^{N,\epsilon}(1)$ is a Gaussian random variable with covariance 
  \[
    Q_F
      = \int_0^\epsilon \pi_F\e^{\nu\pi_F A s}\cov\e^{\nu\pi_F A s}\pi_F\,ds
  \]
  and mean $\pi_F u^{N,\epsilon}(1-\epsilon)$. We denote by $g_{\epsilon,N}$ its
  density with respect to the Lebesgue measure. Since $Q_F$ is bounded and invertible
  on $F$ and its eigenvalues are all of order $\epsilon$, it is easy to see, by
  a simple change of variable (to get rid of the random mean and to extract the
  behaviour in $\epsilon$) and the smoothness of the Gaussian density, that
  \[
   \|g_{\epsilon,N}\|_{B^n_{1,1}}
     \leq c \epsilon^{-\frac{n}2},
  \]
  holds almost surely with a deterministic constant $c>0$.
  
  Fix $\phi\in C^\infty_0(\R^d)$, $n\geq1$ and $h\in\R^d$ with $|h|<1$, then
  \begin{equation}\label{e:split}
    \begin{aligned}
      \E[(\Delta_h^n\phi)(\pi_F u^N(1))]
        &=       \E[(\Delta_h^n\phi)(\pi_F u^N(1)) - (\Delta_h^n\phi)(\pi_F u^{N,\epsilon}(1))]\\
        &\quad + \E[(\Delta_h^n\phi)(\pi_F u^{N,\epsilon}(1))].
    \end{aligned}
  \end{equation}
  Consider the second term and use a discrete integration by parts
  \[
   \begin{aligned}
     \E[(\Delta_h^n\phi)(\pi_F u^{N,\epsilon}(1))]
     & =  \E\big[\E[(\Delta_h^n\phi)(\pi_F u^{N,\epsilon}(1))| \mathcal{F}_{1-\epsilon}]\big]\\
     & =  \E\Big[\int_{\R^d} \Delta_h^n\phi(x) g_{\epsilon,N}(x)\,dx\Big]\\
     & =  \E\Big[\int_{\R^d} \phi(x)\Delta_{-h}^n g_{\epsilon,N}(x)\,dx\Big]\\
     &\le \|\phi\|_{L^\infty}\|h\|^n\E[\|g_{\epsilon,N}\|_{B^n_{1,1}}]\\
     &\le c\|\phi\|_{L^\infty} \epsilon^{-\frac{n}2}\|h\|^n.
    \end{aligned}
  \]
  The first term of~\eqref{e:split} can be estimated as follows
  \[
    \begin{aligned}
      \lefteqn{\bigl|\E[(\Delta_h^n\phi)(\pi_F u^N(1)) - (\Delta_h^n\phi)(\pi_F u^{N,\epsilon}(1))]\bigr|\leq}\\
        &\qquad\leq\Bigl|\sum_{j=0}^n (-1)^{n-j}\binom{n}{j}\E\bigl[\phi\bigl(\pi_F u^N(1)+jh\bigr) - \phi\bigl(\pi_F u^{N,\epsilon}(1)+jh\bigr)\bigr]\Bigr|\\
        &\qquad\leq c[\phi]_\alpha \E\bigl[\|\pi_F\bigl(u^N(1) - u^{N,\epsilon}(1)\bigr)\|^\alpha\bigr],
    \end{aligned}
  \]
  where $[\phi]_\alpha$ is the H\"older semi-norm of $C^\alpha(\R^d)$, and $\alpha\in(0,1)$
  will be suitably chosen later. Since
  \[
    \pi_F\bigl(u^N(1) - u^{N,\epsilon}(1)\bigr)
      = -\int_{1-\epsilon}^1 \e^{-\nu A(1-s)}\pi_F B(u^N(s),u^N(s))\,ds,
  \]
  \eqref{e:maintruc} and \eqref{e:Gbound} yield
  \[
    \E\bigl[\|\pi_F\bigl(u^N(1) - u^{N,\epsilon}(1)\bigr)\|\bigr]
      \leq c_F \int_{1-\epsilon}^1 \E[\|u^N(s)\|_H^2]\,ds
      \leq c_F(\|x\|_H^2 + 1)\epsilon.
  \]
  
 Gathering the estimates of the two terms gives
  \[
    \bigl|\E[(\Delta_h^n\phi)(\pi_F u^N(1))]\bigr|
      \leq c[\phi]_\alpha \epsilon^\alpha + c\epsilon^{-\frac{n}{2}}\|h\|^n\|\phi\|_\infty
      =    c\|\phi\|_{C^\alpha}\|h\|^{\frac{2\alpha n}{2\alpha+n}}
  \]
  where the number $c$ is independent of $N$, and we have chosen
  $\epsilon = \|h\|^{\frac{2 n}{2\alpha+n}}$. By a discrete integration by
  parts,
  \[
    \bigl|\E[(\Delta_{-h}^n\phi)(\pi_F u^N(1))]\bigr|
      = \int_{\R^d} (\Delta_{-h}^n\phi)(x)f_N(x)\,dx
      = \int_{\R^d} (\Delta_h^n f_N)(x)\phi(x)\,dx
  \]
  (we have switched from $h$ to $-h$ for simplicity) and so we have proved
  that for every $h\in\R^d$ with $|h|\leq 1$,
  \begin{equation}\label{e:besovcond}
    \Bigl|\int_{\R^d} \phi(y)\frac{(\Delta_h^n f_N)(x)}{|h|^{\alpha_n}}\,dy\Bigr|
      \leq c\|\phi\|_{C^\alpha},
  \end{equation}
  with $\alpha_n = {\frac{2\alpha n}{2\alpha+n}}$. We wish to deduce from the
  above inequality the following claim:
  \begin{quotation}
    \emph{The sequence $(f_N)_{N\geq n_F}$ is bounded in $B_{1,\infty}^\gamma(\R^d)$
    for every $\gamma\in(0,1)$}.
  \end{quotation}
  Before proving the claim, we show how it immediately implies the statements
  of the theorem. Indeed, since $B_{1,\infty}^\gamma(\R^d)\hookrightarrow L^p(\R^d)$
  for every $p\in[1,\tfrac{d}{d-\gamma}]$ (by \cite[formula 2.2.2/(18)]{Tri83}
  and Sobolev's embeddings), it follows that the sequence $(f_{N_k})_{k\geq1}$
  is uniformly integrable and hence convergent to a positive function $f\in L^p(\R^d)$
  which is the density of $\pi_F u(1)$ (we recall here that
  $\Prob_{N_k}\rightharpoonup\Prob$, where $\Prob$ is the law
  of $u$, hence the limit is unique along the subsequence $(N_k)_{k\geq1}$).
  Moreover, since the bound in the claim is independent of $N$, it follows
  that $f\in B_{1,\infty}^\gamma(\R^d)$ for every $\gamma\in(0,1)$ and
  hence, using again the embedding of Besov spaces into Lebesgue spaces,
  $f\in L^p(\R^d)$ for every $p\in[1,\tfrac{d}{d-1})$.
  
  It remains to show the above claim. Let $\psi\in\mathcal{S}(\R^d)$, where
  $\mathcal{S}(\R^d)$ is the Schwartz space of smooth rapidly decreasing
  functions, and set $\phi = (I-\Delta_d)^{-\beta/2}\psi$, where $\Delta_d$
  is the Laplace operator on $\R^d$ and $\beta>\alpha$ will be suitably chosen
  later. Notice that since $C^\alpha(\R^d)=B_{\infty,\infty}^\alpha(\R^d)$
  \cite[Theorem 2.5.7, Remark 2.2.2/3]{Tri83} and since $(I-\Delta_d)^{-\beta/2}$
  is a continuous operator from $B_{\infty,\infty}^{\alpha-\beta}(\R^d)$ to
  $B_{\infty,\infty}^\alpha(\R^d)$ \cite[Theorem 2.3.8]{Tri83}, it follows that
  \[
    \|\phi\|_{C^\alpha}
      \leq c\|\phi\|_{B_{\infty,\infty}^\alpha}
      \leq c\|\psi\|_{B_{\infty,\infty}^{\alpha-\beta}}
      \leq c_0\|\psi\|_{L^\infty},
  \]
  where the last inequality follows from the fact that
  $L^\infty(\R^d)\hookrightarrow B_{\infty,\infty}^{\alpha-\beta}(\R^d)$,
  since $B_{\infty,\infty}^{\alpha-\beta}(\R^d)$ is the dual of
  $B^{\beta-\alpha}_{1,1}(\R^d)$ \cite[Theorem 2.11.2]{Tri83} and
  $B^{\beta-\alpha}_{1,1}(\R^d)\hookrightarrow L^1(\R^d)$ by definition,
  since $\beta>\alpha$.
  
  Let $g_N = (I-\Delta_d)^{-\beta/2}f_N$, then \eqref{e:besovcond} yields
  \[
    \Bigl|\int_{\R^d} \psi(y)(\Delta_h^n g_N)(y)\,dy\Bigr|
      \leq c_0 |h|^{\alpha_n}\|\psi\|_{L^\infty},
  \]
  hence $\Delta_h^n g_N\in L^1(\R^d)$ and
  \[
    \|\Delta_h^n g_N\|_{L^1}
      \leq c_0 |h|^{\alpha_n}.
  \]
  Moreover, by \cite[Theorem 10.1]{AroSmi61}, $\|g_N\|_{L^1}\leq c\|f_N\|_{L^1}=c$,
  hence $(g_N)_{N\geq n_F}$ is a bounded sequence in $B_{1,\infty}^{\alpha_n}(\R^d)$
  and, since $(I-\Delta_d)^{\beta/2}$ maps $B_{1,\infty}^\alpha(\R^d)$ continuously
  onto $B_{1,\infty}^{\alpha-\beta}(\R^d)$ \cite[Theorem 2.3.8]{Tri83}, it follows that
  $(f_N)_{N\geq n_F}$ is a bounded sequence in $B_{1,\infty}^{\alpha_n-\beta}(\R^d)$
  for every $\beta>\alpha$.

  We notice that by suitably choosing $n\geq 1$, $\alpha\in(0,1)$ and $\beta>\alpha$,
  the number $\alpha_n-\beta$ runs over all reals in $(0,1)$: this can be easily
  seen by noticing that $\alpha_n\to 2\alpha$ as $n\to\infty$. This proves the
  claim and consequently the whole theorem.
\end{proof}
\subsection{Additional regularity for stationary solutions}

We can slightly improve the regularity of densities if we consider a special
class of solutions, namely stationary solutions. Consider again problem
\eqref{e:galerkin}, it admits a unique invariant measure (see for instance
\cite{Fla08}, see also \cite{Rom04} for related results). Denote by $\Prob_N$
the law of the process started at the invariant measure. Every limit point
is a stationary solution of \eqref{e:nseabs}, that is a probability measure
which is invariant with respect to the forward time--shift (other methods
can be used to show existence of stationary solutions, see for instance
\cite{FlaGat95}).

The idea that stationary solutions may have better regularity properties
has been already exploited \cite{FlaRom02,Oda06}.
\begin{theorem}\label{t:besov2}
  Let $F$ be a finite dimensional subspace of $D(A)$ generated by the eigenvalues
  of $A$, namely $F = \Span[e_{n_1}, \dots,e_{n_F}]$ for some arbitrary indices
  $n_1,\dots,n_F$. Let $u$ be a stationary solution of~\eqref{e:nseabs} which
  is a limit point of a sequence of stationary solutions of the spectral Galerkin
  approximation. Under Assumption \ref{a:girsanovnoise}, the projection $\pi_F u(1)$
  has a density $f_F$ with respect to the Lebesgue measure on $F$,
  which is almost everywhere positive.

  Moreover $f_F\in B_{1,\infty}^s(\R^d)$, which in particular implies that
  $f_F\in W^{s,1}(\R^d)$ for every $s\in(0,2)$, where $d = \dim F$.
\end{theorem}
\begin{proof}
  We proceed as in the proof of Theorem~\ref{t:besov1}. Fix a stationary solution
  $u$ with law $\Prob$, a sequence $\Prob_{N_k}\rightharpoonup\Prob$ of
  stationary solutions of \eqref{e:galerkin} and a finite dimensional space
  $F=\Span[e_{n_1},\dots,e_{n_F}]$. Write again
  \begin{equation}\label{e:stima2}
    \begin{aligned}
      \E[(\Delta_h^n\phi)(\pi_F u^N(1))]
        &=       \E[(\Delta_h^n\phi)(\pi_F u^N(1)) - (\Delta_h^n\phi)(\pi_F u^{N,\epsilon}(1))]\\
        &\quad + \E[(\Delta_h^n\phi)(\pi_F u^{N,\epsilon}(1))].
    \end{aligned}
  \end{equation}
  where this time $u^{N,\epsilon}$ is defined as the solution of
  \begin{multline*}
    du^{N,\epsilon}
      + \bigl(\nu Au^{N,\epsilon}
         + B(u^{N,\epsilon}) - (1-\eta_\epsilon)\pi_F B(u^{N,\epsilon}) + {}\\
         + (1-\eta_\epsilon)\pi_F B(\e^{-A(s-1+\epsilon)}u^{N,\epsilon}(1-\epsilon))\bigr)\,ds
      = \pi_N \cov^{\frac12}\,dW_s
  \end{multline*}
  so that again $u^N(t) = u^{N,\epsilon}(t)$ for $t\leq 1-\epsilon$, and for $t\geq 1-\epsilon$
  the process $\pi_F u^{N,\epsilon}$ satisfies
  \[
    dv
      + \bigl(\nu \pi_F A v
      + \pi_F B(\e^{-A(s-1+\epsilon)}u^{N,\epsilon}(1-\epsilon))\bigr)\,ds
      = \pi_F \cov^{\frac12}\,dW_s,
  \]
  which is the same equation as in~\eqref{e:justbefore} with an additional
  adapted external forcing. As before, conditioned to $\mathcal{F}_{1-\epsilon}$, 
  $\pi_F u^{N,\epsilon}(1)$ is Gaussian with covariance $Q_F$. Thus, 
  the second term of~\eqref{e:stima2} has the estimate
  \[
    \bigl|\E[(\Delta_h^n\phi)(\pi_F u^{N,\epsilon}(1))]\bigr|
      \leq c\epsilon^{-\frac{n}2}|h|^n \|\phi\|_\infty.
  \]
  We claim that
  \begin{equation}\label{e:claim3}
    \E\bigl[\|\pi_F u^N(1) - \pi_F u^{N,\epsilon}(1)\|_H\bigr]
      \leq c \epsilon^{\frac32}.
  \end{equation}
  Before proving~\eqref{e:claim3}, we show how to use it to conclude the proof.
  Indeed, as before, the first term on the right--hand side of~\eqref{e:stima2}
  is bounded from above as
  \[
    \begin{aligned}
      \bigl|\E[(\Delta_h^n\phi)(\pi_F u^N(1)) - (\Delta_h^n\phi)(\pi_F u^{N,\epsilon}(1))]\bigr|
        &\leq c[\phi]_\alpha \E\bigl[\|\pi_F u^N(1) - \pi_F u^{N,\epsilon}(1)\|_H\bigr]^\alpha\\
        &\leq c[\phi]_\alpha \epsilon^{\frac32\alpha},
    \end{aligned}
  \]
  and so
  \[
    \int_{R^d} (\Delta_h^n f_N)(x)\phi(x)\,dx
      \leq c[\phi]_\alpha \epsilon^{\frac32\alpha} + c \epsilon^{-\frac{n}2}|h|^n\|\phi\|_\infty
      \leq c \|\phi\|_{C^\alpha} |h|^{\alpha_n},
  \]
  by choosing $\epsilon=|h|^{2n/(3\alpha+n)}$, where this time
  $\alpha_n = \tfrac{3\alpha n}{3\alpha+n}$. As in the proof of
  Theorem~\ref{t:besov1}, the above estimate yields that
  $(f_N)_{N\geq n_F}$ is a bounded sequence in $B_{1,\infty}^s$, for every
  $s<\alpha_n-\alpha$. Since $\alpha_n-\alpha\to 2\alpha$ as $n\to\infty$ and
  $\alpha\in(0,1)$ can be arbitrarily chosen, we conclude that $(f_N)_{N\geq n_F}$
  is bounded in $B_{1,\infty}^s$ for every $s<2$. In particular, since
  $B_{1,\infty}^s(\R^d)\hookrightarrow B_{1,1}^s(\R^d) = W^{s,1}(\R^d)$,
  $(f_N)_{N\geq n_F}$ is also bounded in $W^{s,1}(\R^d)$ for every $s<2$.

  We conclude with the proof of~\eqref{e:claim3}. We have that
  \[
    \pi_F\bigl(u^N(1) - u^{N,\epsilon}\bigr)
      = \int_{1-\epsilon}^1 \e^{-\nu A(1-s)}\pi_F\bigl(B(\e^{-\nu A(s-1+\epsilon)}u^N(1-\epsilon)) - B(u^N(s))\bigr)\,ds,
  \]
  hence by~\eqref{e:maintruc} and H\"older's inequality,
  \begin{equation}\label{e:claim3bis}
    \begin{aligned}
      \lefteqn{\E\bigl[\|\pi_F\bigl(u^N(1) - u^{N,\epsilon}(1)\bigr)\|\bigr]\leq}\\
        &\quad\leq c \int_{1-\epsilon}^1\E\bigl[\bigl(\|\e^{-\nu A(s-1+\epsilon)}u^N(1-\epsilon)\|_H + \|u^N(s)\|_H\bigr)\|\e^{-\nu A(s-1+\epsilon)}u^N(1-\epsilon) - u^N(s)\|_H\bigr]\,ds\\
        &\quad\leq c \int_{1-\epsilon}^1\E\bigl[\bigl(\|u^N(1-\epsilon)\|_H + \|u^N(s)\|_H\bigr)^4\bigr]^{\frac14}
                                           \E\bigl[\|\e^{-\nu A(s-1+\epsilon)}u^N(1-\epsilon) - u^N(s)\|_H^{\frac43}\bigr]^{\frac34}\,ds\\
        &\quad\leq c \int_{1-\epsilon}^1 \E\bigl[\|\e^{-\nu A(s-1+\epsilon)}u^N(1-\epsilon) - u^N(s)\|_H^{\frac43}\bigr]^{\frac34}\,ds,
    \end{aligned}
  \end{equation}
  since $\E[\|u^N(s)\|_H^4]$ is finite, constant in $s$ and uniformly bounded in $N$.
  Now, for $s\in(1-\epsilon,1)$,
  \[
    \e^{-\nu A(s-1+\epsilon)}u^N(1-\epsilon) - u^N(s)
      =   \int_{1-\epsilon}^s \e^{-\nu A(s-r)}B(u^N(r))\,dr
        - \int_{1-\epsilon}^s \e^{-\nu A(s-r)}\cov^{\frac12}dW_r
  \]
  and so
  \[
    \begin{aligned}
      \lefteqn{\E\bigl[\|\e^{-\nu A(s-1+\epsilon)}u^N(1-\epsilon) - u^N(s)\|_H^{\frac43}\bigr]\leq}\\
        &\quad\leq   \E\Bigl[\Bigl(\int_{1-\epsilon}^s\bigl\|\e^{-\nu A(s-r)}B(u^N(r))\bigr\|_H\,dr\Bigr)^{\frac43}\Bigr]
                   + \E\Bigl[\Bigl\|\int_{1-\epsilon}^s \e^{-\nu A(s-r)}\cov^{\frac12}dW_r\Bigr\|_H^{\frac43}\Bigr]\\
        &\quad=\memo{1} + \memo{2}.
    \end{aligned}
  \]
  To estimate \memo{1} we use the inequality
  \[
    \|A^{-\frac12}B(v)\|_H
      \leq c\|v\|_{L^4}^2
      \leq c\|v\|_H^{\frac12}\|v\|_V^{\frac32},
  \]
  standard estimates on analytic semigroups and we exploit the fact that
  $u^N$ is stationary,
  \[
    \begin{aligned}
      \memo{1}
        &\leq \E\Bigl[\Bigl(\int_{1-\epsilon}^s\frac{c}{\sqrt{s-r}}\|u^N(r)\|_H^{\frac12}\|u^N(r)\|_V^{\frac32}\,dr\Bigr)^{\frac43}\Bigr]\\
        &\leq c\epsilon^{\frac13} \E\Bigl[\int_{1-\epsilon}^s\frac1{(s-r)^{\frac23}}\|u^N(r)\|_H^{\frac23}\|u^N(r)\|_V^2\,dr\Bigr]\\
        &=    c\epsilon^{\frac23} \E\bigl[\|u^N\|_H^{\frac23}\|u^N\|_V^2\bigr]\\
        &=    c\epsilon^{\frac23}.
    \end{aligned}
  \]
  The second term is standard,
  \[
    \memo{2}
      \leq \E\Bigl[\Bigl\|\int_{1-\epsilon}^s \e^{-\nu A(s-r)}\cov^{\frac12}dW_r\Bigr\|_H^2\Bigr]^{\frac23}\\
      \leq \Bigl(\frac12\epsilon\Tr(\cov)\Bigr)^{\frac23}
      = c\epsilon^{\frac23},
  \]
  and in conclusion
  \[
    \E\bigl[\|\e^{-\nu A(s-1+\epsilon)}u^N(1-\epsilon) - u^N(s)\|_H^{\frac43}\bigr]
      \leq c\epsilon^{\frac23},
  \]  
  hence from~\eqref{e:claim3bis},
  \[
    \E\bigl[\|\pi_F\bigl(u^N(1) - u^{N,\epsilon}\bigr)\|\bigr]
      \leq c\epsilon^{\frac32},
  \]
  which proves~\eqref{e:claim3}.
\end{proof}
\bibliographystyle{amsplain}

\end{document}